\documentclass[12pt,reqno]{amsart}
\usepackage{a4wide}
\usepackage[english]{babel}
\usepackage[utf8]{inputenc}
\usepackage{amsmath,amsfonts,amssymb,amsthm}
\usepackage{bbm}
\usepackage[table]{xcolor}
\usepackage{graphicx,enumerate,url}
\usepackage[colorinlistoftodos]{todonotes}
\usepackage{footmisc}

\usepackage{tikz}
\usetikzlibrary{calc,through,backgrounds,shapes,matrix,decorations.markings,decorations.pathmorphing,arrows.meta}
\usepackage{caption}

\usepackage{mathrsfs}
\DeclareFontFamily{U}{mathx}{\hyphenchar\font45}
\DeclareFontShape{U}{mathx}{m}{n}{
      <5> <6> <7> <8> <9> <10>
      <10.95> <12> <14.4> <17.28> <20.74> <24.88>
      mathx10
      }{}
\DeclareSymbolFont{mathx}{U}{mathx}{m}{n}
\DeclareFontSubstitution{U}{mathx}{m}{n}
\DeclareMathAccent{\widecheck}{0}{mathx}{"71}

\usepackage{accents}

\usepackage{array}
\usepackage{multirow}
\usepackage{multicol}
\usepackage{hhline}

\usepackage{todonotes}

\usepackage{hyperref}

\usepackage{enumitem}
\setenumerate{label*=\arabic*., itemsep=5pt, topsep=5pt}

\usepackage{url, hypcap}
\hypersetup{colorlinks=true, citecolor=darkblue, linkcolor=darkblue}

\numberwithin{figure}{section}
\numberwithin{equation}{section}

\usepackage{csquotes}

\usepackage{cleveref}

\newtheorem{theorem}{Theorem}[section]
\newtheorem{proposition}[theorem]{Proposition}
\newtheorem{lemma}[theorem]{Lemma}

\newtheorem{conjecture}[theorem]{Conjecture}

\theoremstyle{definition}
\newtheorem{example}[theorem]{Example}
\newtheorem{remark}[theorem]{Remark}

\newtheorem*{open problem}{Open Problem}

\newcommand{\ZZ}{\mathbb{Z}}

\newcommand{\RR}{\mathbb{R}}

\newcommand{\set}[2]{\left\{ #1 \;|\; #2 \right\}}

\newcommand{\Sym}{{\mathfrak{S}}}

\newcommand{\des}{\operatorname{des}}
\newcommand{\Des}{\operatorname{Des}}

\newcommand{\Chow}{\operatorname{H}}
\newcommand{\Chowaug}{\operatorname{H}^{\operatorname{aug}}}
\newcommand{\p}[4]{p_{#1,#2}^{#3 \subseteq #4}}
\newcommand{\pSempty}[3]{p_{#1,#2}^{#3}}

\newcommand{\interlace}{\preceq}
\newcommand{\spaceinterlace}{\ \preceq \ }
\newcommand{\quadinterlace}{\quad \preceq \quad }

\newcommand{\rk}{\operatorname{rk}}
\newcommand{\Ptop}{{\widehat{P}}}

\newcommand{\Dfn}[1]{\emph{\bfseries #1}}

\definecolor{darkblue}{rgb}{0,0,0.7}

\definecolor{lightblue}{rgb}{0.68,0.85,1}

\definecolor{lightgrey}{rgb}{0.9,0.9,0.9}

\definecolor{grey}{rgb}{0.5,0.5,0.5}

\definecolor{darkgreen}{RGB}{0,128,0}

\title[Chow polynomials of simplicial posets with positive $h$-vector]{Chow polynomials of simplicial posets\\ with positive $h$-vector are real-rooted}

\author[E.~Hoster]{Elena Hoster}
\author[C.~Stump]{Christian Stump}
\address[E.~Hoster \& C.~Stump]{Fakultät für Mathematik, Ruhr-Universität Bochum, Germany}
\email{\{elena.hoster,christian.stump\}@rub.de}

\begin{document}
    
\begin{abstract}
    We prove that a finite  graded simplicial poset with a top element added has real-rooted Chow and augmented Chow polynomials whenever it has a positive $h$-vector.
    This class of posets include Cohen-Macaulay simplicial posets and in particular lattices of flats of uniform matroids.
\end{abstract}

\maketitle


\section{Main Results}
\label{sec: intro}
Let~$P$ be a finite graded simplicial poset.
This is,~$P$ is a finite poset with $\hat 0$ for which all maximal intervals are boolean of the same rank~$n$.
Its $h$-vector $(h_0,\dots,h_n)$ is given by $h_0 + h_1 x + \dots + h_n x^n = \sum_{i=0}^{n} f_{i-1}x^i(1-x)^{n-i}$ where $f_{i-1}$ is the number of elements of rank~$i$~\cite[III.6]{stanleyCombComm}, and we say that~$P$ has a \Dfn{positive $h$-vector} if $h_i \geq 0$ for all~$i$.
Let~$\Ptop$ be obtained from~$P$ by adding a top element~$\hat 1$.
Following~\cite{ferroni2024chowfunctionspartiallyordered}, the Chow and augmented Chow polynomials of $\Ptop$ are
\begin{align}
    \Chow_{\Ptop}(x) &= \sum_{\substack{ S \subseteq \{2,\dots,n\} \\ S \text{ isolated} }}
\beta_{\Ptop}(S) \cdot x^{|S|} \cdot (1+x)^{n - 2 \cdot |S| } \,,\\[5pt]
    \Chowaug_{\Ptop}(x) &= \sum_{\substack{ S \subseteq \{1,\dots,n\} \\ S \text{ isolated}}}
\beta_{\Ptop}(S) \cdot x^{|S|} \cdot (1+x)^{n + 1 - 2 \cdot |S| } \,.
\end{align}
Here, a set $S \subset \ZZ$ is \Dfn{isolated}\footnote{This property has been called \emph{good} in~\cite{ferroni2024chowfunctionspartiallyordered} and $\operatorname{Stab}(\ZZ)$ in~\cite{MR3878174}.} if $i \in S$ implies $i+1 \notin S$, and $\beta_\Ptop(S)$ of $S \subseteq \{1,\dots,n\}$ denotes the \Dfn{flag $h$-vector}
\[
    \beta_{\Ptop}(S) = \sum_{T \subseteq S} (-1)^{|S \setminus T|} \alpha_{\Ptop}(T)\,,
\]
for $\alpha_{\Ptop}(T)$ being the \Dfn{flag $f$-vector} counting maximal chains in the subposet of~$\Ptop$ with only the ranks in~$T$ selected.

\begin{theorem}\label{thm: main thm chows are rr}
    Let $P$ be a finite graded simplicial poset with positive $h$-vector.
    Then
    \[
    \Chow_{\Ptop}(x) \text{ and }\Chowaug_{\Ptop}(x)\text{ are real-rooted.}
    \]
\end{theorem}

For the Chow polynomials of the dual poset $\Ptop^\ast$ of $\Ptop$, we even have the analogous statement and in addition the interlacingness.

\begin{theorem}\label{thm: dual chow rr}
    Let $P$ be a finite graded simplicial poset with positive $h$-vector.
    Then
    \[
    \Chow_{\Ptop^\ast}(x) \text{ is real-rooted.}
    \]
    Moreover, the roots of $\Chow_{\Ptop^\ast}(x)$ interlace the roots of $\Chowaug_{\Ptop^\ast}(x)= \Chowaug_{\Ptop}(x)$.
\end{theorem}

\begin{example}
\label{ex:U34}
    Let $M:=U_{3,4}$ be the uniform matroid of rank $3$ on $E=\{1,2,3,4\}$.  
    Denote by $\mathcal L(M)$ its lattice of flats
    
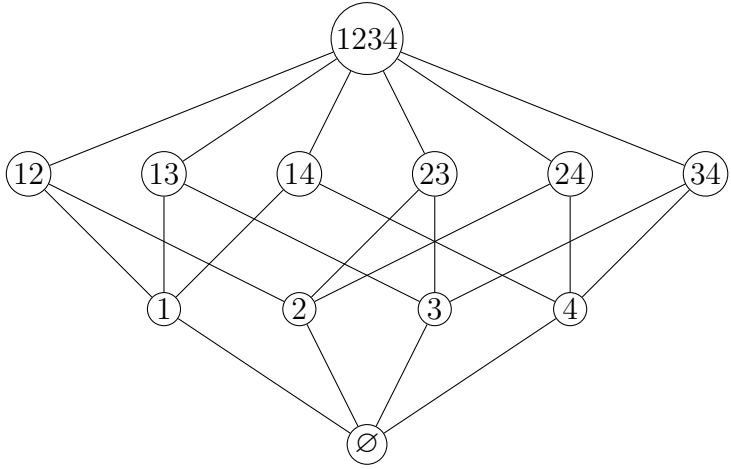
\begin{center}
        \begin{tikzpicture}[scale=0.9,
                every node/.style={draw, circle, inner sep=1pt},
                baseline=(E.base)]
          \node (0) at (0,0) {$\varnothing$};
        
          \node (1) at (-3,2) {$1$};
          \node (2) at (-1,2) {$2$};
          \node (3) at ( 1,2) {$3$};
          \node (4) at ( 3,2) {$4$};
        
          \node (12) at (-5,4) {$12$};
          \node (13) at (-3,4) {$13$};
          \node (14) at (-1,4) {$14$};
          \node (23) at ( 1,4) {$23$};
          \node (24) at ( 3,4) {$24$};
          \node (34) at ( 5,4) {$34$};
        
          \node (E)  at (0,6) {$1234$};
        
          \foreach \x in {1,2,3,4}
            \draw (0)--(\x);
        
          \draw (1)--(12) (1)--(13) (1)--(14);
          \draw (2)--(12) (2)--(23) (2)--(24);
          \draw (3)--(13) (3)--(23) (3)--(34);
          \draw (4)--(14) (4)--(24) (4)--(34);
        
          \foreach \x in {12,13,14,23,24,34}
            \draw (\x)--(E);
        \end{tikzpicture}
    \end{center}
    and let~$P = \mathcal L(M)\setminus\{E\}$, so that $\Ptop = \mathcal L(M)$.
    Then~$P$ is clearly simplicial of rank~$n=2$ with the desired properties and we obtain the flag $f$-vector
    \[
      \alpha_{\Ptop}(\varnothing)=1,\quad
      \alpha_{\Ptop}(\{1\})=4,\quad
      \alpha_{\Ptop}(\{2\})=6,\quad
      \alpha_{\Ptop}(\{1,2\})=12
    \]
    and the flag $h$-vector
    \[
      \beta_{\Ptop}(\varnothing)=1,\quad
      \beta_{\Ptop}(\{1\})=3,\quad
      \beta_{\Ptop}(\{2\})=5,\quad
      \beta_{\Ptop}(\{1,2\})=3\,.
    \]
    Moreover, the $f$-vector of~$P$ is $(1,4,6)$ and the $h$-vector is $(1,2,3)$\,.
    The flag $h$-vector of $\Ptop^\ast$ is given 
    by~$\beta_{\Ptop^\ast}(S)=\beta_{\Ptop}(\set{3-s}{s\in S})$.
    Considering the isolated subsets of $\{2\}$ and of $\{1,2\}$, respectively, we get
    \begin{align*}
      \Chow_{\Ptop}(x)
        &=\beta_\Ptop(\varnothing) \cdot (1+x)^2+\beta_\Ptop({\{2\}})\,x
          =(1+x)^2+5x
          =x^2+7x+1 \\
          &= \left(x+\tfrac{7-3\sqrt5}{2}\right)\cdot\left(x+\tfrac{7+3\sqrt5}{2}\right),\\[6pt]
      \Chow_{\Ptop^\ast}(x)
        &=\beta_{\Ptop^\ast}(\varnothing) \cdot (1+x)^2+\beta_{\Ptop^\ast}({\{2\}})\,x
          =(1+x)^2+3x
          =x^2+5x+1 \\
          &= \left(x+\tfrac{5-\sqrt{21}}{2}\right)\cdot\left(x+\tfrac{5+\sqrt{21}}{2}\right),\\[6pt]
      \Chowaug_{\Ptop}(x)
        &=\beta_\Ptop(\varnothing) \cdot (1+x)^3
          +\biggl(\beta_\Ptop({\{1\}})+\beta_\Ptop({\{2\}})\biggr)x(1+x)\\
        &= (1+x)^3+3x(1+x)+5x(1+x)
          =x^3+11x^2+11x+1\\
        &=(x+1)\cdot(x+5-2\sqrt6)\cdot(x+5+2\sqrt6)\,.
    \end{align*}
    We finally observe that
    \[
      -5-2\sqrt6 < \tfrac{-7-3\sqrt5}{2} < \tfrac{5-\sqrt{21}}{2} < -1 
      < \tfrac{5+\sqrt{21}}{2} < \tfrac{-7+3\sqrt5}{2} < -5+2\sqrt6\,,
    \]
    so the roots of $\Chow_{\Ptop^\ast}(x)$ indeed interlace the roots of $\Chowaug_\Ptop(x)$. 
    Moreover, we observe that also the roots of $\Chow_{\Ptop}(x)$ interlace the roots of $\Chowaug_\Ptop(x)$, compare \Cref{conj: interlace and neg h} below.
\end{example}

For later usage, we record a relationship between the flag $h$-vector of $\Ptop$ and the $h$-vector of~$P$ that holds for simplicial posets~$P$, namely
\begin{equation}
    \beta_{\Ptop}(S)
     = \sum_{k=0}^{n} h_{k} \cdot
     \#\big\{w\in\mathfrak{S}_{n+1} \mid w(1)=k+1, \ \Des(w)=n+1-S \big\}\,,
\label{eq:flaghsimplicial}
\end{equation}
where $n+1-S = \set{n+1-s}{s\in S}$.

\begin{example}
    We continue the above example, where we have $n=2$.
    We thus obtain
    \begin{align*}
        \beta_\Ptop(\varnothing) &= 1 \cdot \#\{123\} = 1 \\
        \beta_\Ptop(\{1\})   &= 1 \cdot \#\{132\} + 2 \cdot \#\{231\} = 3 \\
        \beta_\Ptop(\{2\})   &= 2 \cdot \#\{213\} + 3 \cdot \#\{312\} = 5 \\
        \beta_\Ptop(\{1,2\}) &= 3 \cdot \#\{321\} = 3 \\
    \end{align*}
\end{example}

The relation \eqref{eq:flaghsimplicial} will later ensure the real-rootedness of Chow polynomials when the $h$-vector is positive.
However, we were not able to find a counterexample otherwise.

\begin{conjecture}\label{conj: interlace and neg h}
    Let $P$ be a simplicial poset. Then
    \[
        \Chow_\Ptop (x),\, \Chow_{\Ptop^\ast} (x) \text{ and } \Chowaug_\Ptop (x) \text{ are real-rooted.}
    \]
    Moreover, the roots of both $\Chow_\Ptop (x)$ and $\Chow_{\Ptop^\ast} (x)$ interlace the roots of $\Chowaug_\Ptop (x) = \Chowaug_{\Ptop^\ast} (x)$.
\end{conjecture}
Conjectures about the interlacing property had also been made by other authors, see e.g. \cite[Remark 4.28]{ferroni2024chowfunctionspartiallyordered}.
We conclude this section by relating our main result to existing results in the literature.

\subsubsection*{Chow polynomials for uniform matroids}

The lattice of flats of the uniform matroid $U_{k,n}$ has coatoms given by all $(k-1)$-subsets of $\{1,\dots,n\}$.
It is thus obtained from a graded, simplical poset by adding a top element.
Our main result thus recovers~\cite[Theorem~1.10]{FMSV2024} where it was shown that its augmented Chow polynomial is real-rooted and~\cite{braendenvecchi2025} where it was shown that its Chow polynomial is real-rooted.

\subsubsection*{Chain polynomials of simplicial posets}

In~\cite{athanasiadis2024classesposetsrealrootedchain} it was shown that the chain polynomial of $\Ptop$ is real-rooted, if $P$ is a finite graded simplical poset with positive $h$-vector.
Our argument for the real-rootedness of the Chow polynomials is based on the same initial rewriting of the flag $h$-vector in~\Cref{eq:flaghsimplicial},
but the families of interlacing polynomials we use is rather different from theirs.
Our approach relies on a decomposition of Chow polynomials into palindromic polynomials using the $\gamma$-expansion. 
In particular, for the Boolean lattice $\Ptop = B_n$ of rank $n$, this yields a decomposition of both the Eulerian polynomials and the binomial Eulerian polynomials into interlacing palindromic polynomials, which provides an alternative proof of their real-rootedness.

\subsubsection*{Chow polynomials of Cohen-Macaulay posets}

It is conjectured in~\cite[Conjecture~4.26]{ferroni2024chowfunctionspartiallyordered} that the Chow polynomials of Cohen-Macaulay posets are real-rooted.
Our main results proves this conjecture in the case of Cohen-Macaulay simplicial posets as Cohen-Macaulay posets are known to be $h$-positive~\cite[Theorem 6.4]{stanleyCombComm}.
It is moreover conjectured in~\cite[Remark 4.28]{ferroni2024chowfunctionspartiallyordered} that the Chow polynomial of a Cohen-Macaulay poset interlaces the augmented Chow polynomial.
Our main result again proves this conjecture in the case of the dual of Cohen-Macaulay simplicial posets.

\subsubsection*{Chow and chain polynomials of geometric lattices}

Geometric lattices are the lattices of flats of matroids.
It is conjectured that Chow polynomials of geometric lattices are real-rooted~\cite[Conjecture~8.18]{ferroni2023valuativeinvariantslargeclasses} and that chain polynomial of geometric lattices are real-rooted~\cite[Conjecture~1.2]{MR4565299}.
Our main result covers only very particular cases of the first conjecture.

\section{Properties of real-rooted polynomials}
\label{subsec: rr}

A polynomial is \Dfn{real-rooted} if it has only real roots.
This section recalls properties of real-rooted polynomials from~\cite{fisk2008polynomialsrootsinterlacing} and~\cite{brändén2014unimodalitylogconcavityrealrootedness}.
\[
  f = \prod(x-\alpha_i),\quad g = \prod(x-\beta_i) \in \RR_{\geq 0}[x] 
\]
be two real-rooted polynomial with nonnegative coefficients\footnote{Many of the properties referenced in this paper originally only ask for a positive leading term.} and weakly decreasing roots
\[ 
\alpha_1 \geq \alpha_2 \geq \dots  \qquad \text{and} \qquad \beta_1 \geq \beta_2 \geq \dots \,.
\]
Then $f$ \Dfn{interlaces} $g$, denoted by $f \interlace g$ if $\deg g \in \{ \deg f , \deg f +1\}$ and the roots of~$f$ interlace the roots of~$g$, 
\[
  \beta_1 \geq \alpha_1 \geq \beta_2 \geq \alpha_2 \geq \dots \,.
\] 
By convention, any two polynomials of degree 0 or 1 interlace, $0\interlace 0$, $0 \interlace f$, and $f\interlace 0$ for all real-rooted polynomials $f$.
We collect the following elementary properties for sums of interlacing polynomials, see for example~\cite[Chapter 3]{fisk2008polynomialsrootsinterlacing}.

\begin{lemma}
\label{lma: interlacing sums}
    Let $f,g,h \in \RR_{\geq 0}[x]$ be real-rooted polynomials.
    \begin{enumerate}
        \item If $f,g \interlace h$, then $f+g \interlace h$.
        \item If $f \interlace g,h$, then $f \interlace g+h$.
        \item If $f \interlace g$, then $g \interlace x \cdot f$.
    \end{enumerate}
\end{lemma}
The relation of interlacing polynomials is not transitive and a sequence~$(f_1,\dots, f_n)$ is an \Dfn{interlacing sequence} if all polynomials are real-rooted, and $f_i \interlace f_j$ for $i < j$.

\begin{lemma}[\cite{WAGNER1992459}]
\label{lma: interlacing seq}
    Let $f_1, f_2, \dots, f_n \in \RR_{\geq 0}[x]$ be real-rooted polynomials, $f_i \not\equiv 0$ for all $i$.
    Then~$(f_1,\dots , f_n)$ is an interlacing sequence if and only if $f_i \interlace f_{i+1}$ for~$1 \leq i<n$ and $f_1 \interlace f_n$.
\end{lemma}

Interlacing sequences play a fundamental role in proving real-rootedness.
For later usage, we list further fundamental properties for sums of interlacing polynomials, 
following from~\cite[Theorem 8.5]{brändén2014unimodalitylogconcavityrealrootedness}.

\begin{lemma}
\label{lma: recycle interlacing seq}
    Let $(f_1,\dots , f_n)$ be an interlacing sequence.
    \begin{enumerate}
        \item Let $\lambda_1,\dots,\lambda_n \in \RR_{\geq 0}$, 
        then $\sum_i \lambda_i f_i$ is real-rooted and 
        \[
            f_1 \spaceinterlace \sum_{i = 1}^n \lambda_i f_i \spaceinterlace f_n \,.
        \]
        \item The sequence $(p_1, \dots , p_n)$ of lower partial sums,
        \[
        p_k(x) = \sum_{i = 1}^{k} f_i(x) \quad \text{ for } 1\leq k \leq n \,,
        \]
        is an interlacing sequence.
        \item The sequence $(q_1, \dots , q_n)$ of upper partial sums,
        \[
        q_k(x) = \sum_{i = k}^{n} f_i(x) \quad \text{ for } 1\leq k \leq n \,,
        \]
        is an interlacing sequence.
        \item The sequence $(r_1, \dots , r_{n-\ell+1})$ defined by
        \[
        r_k(x) = f_k(x) + \dots + f_{k + \ell }(x) \quad \text{ for } 1\leq k \leq n-\ell \,,
        \] 
        is an interlacing sequence for any $0 \leq \ell < n$.
        \item The sequence $(t_1, t_2, \dots , t_{n+1})$ defined by
        \[
        t_k(x) = x \cdot \sum_{i = 1}^{k-1} f_i(x) + \sum_{i=k}^{n} f_i(x) \quad \text{ for } 1\leq k \leq n+1 \,,
        \]
        is an interlacing sequence.
    \end{enumerate}
\end{lemma}

A polynomial $f(x) = \lambda_0 + \lambda_1 x + \dots \lambda_d x^d \in \RR_{\geq 0}[x]$ of degree $d$ is \Dfn{palindromic} if $\lambda_i = \lambda_{d-i}$ for all~$0\leq i \leq d$.
We next turn to the \Dfn{$\gamma$-expansion} of such a palindromic polynomial~$f$ given by
\[
    f(x) = \sum_{i = 0}^{\lfloor d/2 \rfloor} \gamma_i\cdot x^i\cdot (1+x)^{d-2i}
\]
and its associated \Dfn{$\gamma$-polynomial}
\[
    \gamma(f ; x) = \sum_{i = 0}^{\lfloor d/2 \rfloor} \gamma_i\cdot x^i \,.
\]

Real-rootedness interacts naturally with $\gamma$-expansions.

\begin{lemma}[{\cite[Observation~4.2]{PetersenEulerianNumbers}}]
\label{lma: gamma exp rr}
    Let $f \in \RR_{\geq 0}[x]$ be palindromic.
    Then
    \[
      f \text{ real-rooted} \quad \Longleftrightarrow\quad \gamma (f) \text{ real-rooted}\,.
    \]
\end{lemma}

The following proposition extends this to interlacingness.

\begin{proposition}
\label{prop: f interlaces iff gamma}
    Let $f,g \in \RR_{\geq 0}[x]$ be palindromic with $\deg g = \deg f + 1$.
    Then
    \[
      f \interlace g \quad \Longleftrightarrow \quad \gamma (f) \interlace \gamma (g)\,.
    \]
\end{proposition}

\begin{proof}
    We first consider a general palindromic and real-rooted polynomial $h(x) \in \RR_{\geq 0}[x]$ of degree~$d$.
    Its roots are
    \[ 
        0 > \rho_1 \geq \dots \geq \rho_d
    \]
    with $\rho_i = 1/\rho_{k+1-i}$ for all~$i$, and the roots of $\gamma (h)$ are
    \begin{equation}
        \rho_i / (1+\rho_i)^2 \text{ for } 0 > \rho_i > -1\,.
    \label{eq:gammaroots}
    \end{equation}
    Moreover, the multiplicity of the root $-1$ in~$h$ is
    \begin{equation}
        \mu_h(-1) = d - 2\cdot\deg \gamma (h)\,.
    \label{eq:gammamulti}
    \end{equation}

    \medskip
    By \Cref{lma: gamma exp rr}, we may assume that $f,g,\gamma (f),\gamma (g)$ are all real-rooted and the relation between the roots is as given in~\eqref{eq:gammaroots} and the multiplicities of $-1$ are given in~\eqref{eq:gammamulti}.
    Let
    \[
      f = \prod_{i=1}^d(x-\alpha_i),\quad g = \prod_{i=1}^{d+1}(x-\beta_i)
    \]
    with
    \[ 
        0 > \alpha_1 \geq \dots \geq \alpha_d  \qquad \text{and} \qquad 0 > \beta_1 \geq \dots \geq \beta_{d+1}\,.
    \]

    \medskip
    We first assume that $f \interlace g$.
    Since the function $\lambda \mapsto \tfrac{\lambda}{(1+\lambda)^2}$ is strictly increasing in the open interval $(-1,0)$, the roots of $\gamma (f)$ and of $\gamma (g)$ interlace in the same order as the roots of $f$ and of $g$ interlace in this interval.
    We conclude $\gamma (f) \interlace \gamma (g)$.

    \medskip
    We next assume that $\gamma (f) \interlace \gamma (g)$.
    By the same argument as above, the roots of $\gamma (f)$ and of $\gamma (g)$ interlace in the same order as the roots of $f$ and of $g$ inside the open interval $(-1,0)$.
    We now distinguish the two cases $\deg \gamma (g) \in \{\deg \gamma (f), \deg\gamma (f)+1\}$.

    If $\deg \gamma (g) = \deg \gamma (f) = k$, then we have found $2k$ roots of~$f$ and of~$g$, and $\deg g = \deg f + 1$ implies that $\mu_g(-1) = \mu_f(-1) + 1$.
    We conclude that the roots of~$f$ interlace the roots of~$g$, $f \interlace g$.

    If $\deg \gamma (g) = \deg \gamma (f) + 1 = k$, the we have found $2k$ roots of~$g$ and $2k-2$ roots of~$f$, and $\deg g = \deg f + 1$ implies that $\mu_g(-1) = \mu_f(-1) - 1$.
    We again conclude that the roots of~$f$ interlace the roots of~$g$, $f \interlace g$.
\end{proof}

\begin{example}
    We continue the running example for $\Ptop = \mathcal{L}(U_{3,4})$ and compute
    \begin{align*}
      \Chow_\Ptop(x)        &= 1\cdot x^0\cdot(1+x)^2 + 5\cdot x\cdot(1+x)^0\\
      \Chow_{\Ptop^\ast}(x) &= 1\cdot x^0\cdot(1+x)^2 + 3\cdot x\cdot(1+x)^0\\
      \Chowaug_{\Ptop}(x)   &= 1\cdot x^0\cdot(1+x)^3 + 8\cdot x\cdot(1+x)
    \end{align*}
    so we obtain
    \[
      \gamma(\Chow_\Ptop ; x ) = 1+5\cdot x\,, \quad
      \gamma(\Chow_{\Ptop^\ast} ; x ) = 1+3\cdot x\,, \quad\text{and}\quad
      \gamma(\Chowaug_{\Ptop} ; x ) =1+8\cdot x\,,
    \]
    which interlace as $\gamma(\Chow_\Ptop ; x ) \interlace \gamma(\Chowaug_{\Ptop} ; x )$ 
    and $\gamma(\Chow_{\Ptop^\ast} ; x ) \interlace \gamma(\Chowaug_{\Ptop} ; x )$.
\end{example}

\section{A family of interlacing polynomials}
\label{sec: interlacing family}

In this section, we introduce a family of interlacing polynomials that play the crucial role in the proof of \Cref{thm: main thm chows are rr}.
Let $S\subseteq T \subseteq \{1,\dots , n\}$.
Define the polynomial $p_{n,k}^{S\subseteq T}\in \RR_{\geq 0}[x]$ for $0\leq k\leq n$ by
\begin{align}
    \label{eq: def p}
    p_{n,k}^{S\subseteq T} (x)
    = \sum_{\substack{w\in \widehat{\Sym}_{n+1,k+1} \\ S\subseteq \Des(w) \subseteq T}} x^{\des(w)} \,,
\end{align}
where $\widehat{\Sym}_{n+1,k+1} = \set{w\in \Sym_{n+1} }{w(1) = k+1 \,, \ \Des(w) \ \text{isolated}} $.
When the set~$S$ is isolated, this polynomial can be described recursively by
\begin{align*}
    \p{1}{k}{S}{T}(x) = \begin{cases}
        1   & k=0 \text{ and } S = \emptyset \,,\\
        x   & k = 1 \text{ and } T = \{1\} \,,\\
        0   & \text{else} \,,
    \end{cases}
\end{align*}
and 
\begin{align*}
    \p{n}{k}{S}{T}(x) =
        \mathbbm{1}_{ \{1\in T\} } \cdot x \cdot \sum_{j = 0}^{k-1} \p{n-1}{j}{S-1}{(T-1)\setminus \{1\}} (x)   
        +
        \mathbbm{1}_{ \{1\notin S\} } \cdot 
        \sum_{j = k}^{n-1} \p{n-1}{j}{S-1}{T-1} (x)     \,.
\end{align*}
where $S-1 = \set{i-1}{i\in S}\setminus \{1\}$ and $T-1$ analogously.
For $s\in S\subseteq T$, we have 
    \begin{align*}
    \p{n}{k}{S\setminus\{s\}}{T} (x) &= \p{n}{k}{S}{T} (x) + \p{n}{k}{S\setminus\{s\}}{T\setminus\{s\}} (x) \,,
    \end{align*}
since the left side polynomial counts permutations in which $s$ \textbf{can be a descent}, 
while the polynomials on the right side count permutations in which $s$ \textbf{must be a descent} and in which $s$ \textbf{must not be a descent}, respectively.

\smallskip
We write~$\pSempty{n}{k}{T} = \p{n}{k}{\emptyset}{T}$ and~$[a,b] = \{a,a+1,\dots,b\}$.
The following lemma gives a description of the Chow polynomials in terms of the polynomials $\p{n}{k}{S}{T}$ for a particular choice of subsets.

\begin{lemma}\label{lma: gamma Chow via abc}
    Let $P$ be a graded simplicial poset of rank $n\geq 2$ with $h$-vector~$(h_0,\dots , h_n)$.
    Let $\Ptop^\ast$ be the dual poset of $\Ptop$. 
    Then,
    \begin{align*}
        \gamma(\Chow_{\Ptop} ; x ) 
        &= \sum_{k = 0}^{n} 
            h_k \cdot \pSempty{n}{k}{[1,n-1]}(x) \,, \\
        \gamma (\Chowaug_{\Ptop^\ast} ; x )  
        &= \sum_{k = 0}^{n} 
        h_k \cdot \pSempty{n}{k}{[2,n]}(x) \,, \quad \text{and} \\
        \gamma(\Chowaug_\Ptop ; x) 
        = \gamma (\Chowaug_{\Ptop^\ast} ; x )  
        &=\sum_{k = 0}^{n} 
        h_k \cdot \pSempty{n}{k}{[1,n]}(x) \,.
    \end{align*}
\end{lemma}
\begin{proof}
    We use \Cref{eq:flaghsimplicial} to rewrite the flag $h$-vectors $\beta_\Ptop(S)$ in terms of the $h$-vector of~$P$.
    The $\gamma$-polynomial of the Chow polynomial then becomes
    \begin{align*}
        \gamma(\Chow_{\Ptop} ; x ) 
        =& \sum_{\substack{ S \subseteq \{2,\dots , n\} \\ S \text{ isolated} }} \beta_\Ptop (S) \cdot x^{|S|} \\
        =& \sum_{k = 0}^{n} 
        h_k
        \sum_{\substack{ S \subseteq \{2,\dots , n\} \\ S \text{ isolated} }}
        \# \set{w\in \Sym_{n+1}}{w(1) = k+1\,, \ \Des(w) = n+1-S } \cdot x^{|S|} \\
        =& \sum_{k = 0}^{n} h_{k}
        \sum_{\substack{S\subseteq \{1,\dots , n-1\} \\ S \text{ isolated} }} 
        \# \set{w \in \Sym_{n+1}}{w(1) = k+1 \,, \ \Des(w) = S} 
        \cdot x^{|S|} \\
        =& \sum_{k = 0}^{n} h_{k}
        \sum_{\substack{w \in \widehat{\Sym}_{n+1,k+1} \\ n \notin \Des(w) }} 
         x^{\des(w) } \\
        =& \sum_{k = 0}^{n} h_{k} \cdot \pSempty{n}{k}{[1,n-1]} (x) \,.
    \end{align*}
    The calculation for the augmented Chow polynomial is completely analogous.
    %
    The equalities for the dual poset $\Ptop^\ast$ now follow immediately with $\beta_{\Ptop^\ast}(n+1-S) = \beta_\Ptop(S)$.
\end{proof}

\begin{example}
    Let $\Ptop$ be the lattice of flats of $U_{3,4}$ from our running example.
    The $h$-vector of $P$ is $(h_0,h_1,h_2) = (1,2,3)$.
    The polynomials $\pSempty{2}{k}{T}$ are
    \begin{align*}
        \pSempty{2}{0}{T}(x) 
        &= \pSempty{1}{0}{T-1}(x) + \pSempty{1}{1}{T-1}(x) 
        = 1 + x \cdot \mathbbm{1}_{\{2\in T\}} \,, \\
        \pSempty{2}{1}{T}(x) 
        &= \mathbbm{1}_{ \{1\in T\} } \cdot x\cdot \pSempty{1}{0}{(T\setminus\{2\})-1}(x) 
        + \pSempty{1}{1}{T-1}(x) 
        = \mathbbm{1}_{ \{ 1\in T \} } \cdot x + \mathbbm{1}_{ \{ 2\in T \} } \cdot x \,,\\
        \pSempty{2}{2}{T}(x) 
        &= \mathbbm{1}_{ \{ 1\in T \} } \cdot x \cdot \left( \pSempty{1}{0}{(T\setminus\{2\})-1}(x) + x\cdot \pSempty{1}{1}{(T\setminus\{2\})-1}(x) \right)
        = \mathbbm{1}_{ \{ 1\in T \} } \cdot x \,.
    \end{align*}
    For the supsets $T = \{1\}$ and $T=\{1,2\}$, this is
    \begin{center}
        \begin{tabular}{l|lll}
        $T$         & $\pSempty{2}{0}{T}$   & $\pSempty{2}{1}{T}$   & $\pSempty{2}{2}{T}$ \\ \hline
        $\{1\}$     & $1$                   & $x$                   & $x$ \\ \hline
        $\{2\}$     & $1+x$                 & $x$                   & $0$ \\ \hline
        $\{1,2\}$   & $1+x$                 & $2x$                  & $x$
        \end{tabular}
    \end{center}
    Thus \Cref{lma: gamma Chow via abc} yields
    \begin{align*}
        \gamma(\Chow_\Ptop ; x)     
            &= 1\cdot \hspace*{15.2pt} 1 \hspace*{15.2pt} + 2\cdot \phantom{2}x + 3 \cdot x 
            = 1 + 5x \,, \\
        \gamma (\Chowaug_{\Ptop^\ast} ; x )      
            &= 1\cdot (1+x) + 2\cdot \phantom{2}x + 3 \cdot 0 \hspace*{0.7pt}
            = 1 + 3x \,, \\
        \gamma (\Chowaug_{\Ptop} ; x )   
            &= 1\cdot (1+x) + 2\cdot 2x + 3 \cdot x 
            = 1 + 8x \,.
    \end{align*}
\end{example}

The following theorem is the hard of the argument.

\begin{theorem}\label{thm: interlacing diag}
    Let $n\geq 2$ and let $T = [1,n]$ or $T=[1,n-1]$.
    The following diagram is an interlacing diagram,
    that is, 
    any path $(\cdot \rightarrow \cdot \rightarrow \cdots \rightarrow \cdot \rightarrow \cdot )$ forms an interlacing sequence.
    \begin{center}
        \begin{tikzpicture}[node distance=0.4cm and 0.6cm]
            \node (a0) at (0,0) {$\pSempty{n}{0}{T\setminus \{1\}}$};
            \node (a1)  [right=of a0] {$\pSempty{n}{1}{T\setminus \{1\}}$};
            \node (dots1)   [right=of a1]   {$\cdots$};
            \node (ak)   [right=of dots1]   {$\pSempty{n}{k}{T\setminus \{1\}}$};
            \node (dots2)  [right=of ak]   {$\cdots$};
            \node (an1)  [right=of dots2] {$\pSempty{n}{n-1}{T\setminus \{1\}}$};
            \node (an)  [right=of an1] {$\pSempty{n}{n\phantom{-1}}{T\setminus \{1\}}$};
            
            \node (c0) [below=of a0] {$\pSempty{n}{0}{T}$};
            \node (c1) [below=of a1] {$\pSempty{n}{1}{T}$};
            \node (dots3)   [right=of c1]   {$\cdots$};
            \node (ck)   [below=of ak]   {$\pSempty{n}{k}{T}$};
            \node (dots4)  [right=of ck]   {$\cdots$};
            \node (cn1)  [below=of an1] {$\pSempty{n}{n-1}{T}$};
            \node (cn)  [below=of an] {$\pSempty{n}{n\phantom{-1}}{T}$};
            
            \node (b0) [below=of c0] {$\p{n}{0}{\{1\}}{T}$};
            \node (b1) [below=of c1] {$\p{n}{1}{\{1\}}{T}$};
            \node (dots5)   [right=of b1]   {$\cdots$};
            \node (bk)   [below=of ck]   {$\p{n}{k}{\{1\}}{T}$};
            \node (dots6)  [right=of bk]   {$\cdots$};
            \node (bn1)  [below=of cn1] {$\p{n}{n-1}{\{1\}}{T}$};
            \node (bn)  [below=of cn] {$\p{n}{n\phantom{-1}}{\{1\}}{T}$};
            
            \draw[->] (a0) -- (a1);
            \draw[->] (a1) -- (dots1) -- (ak);
            \draw[->] (ak) -- (dots2) -- (an1);
            \draw[->] (an1) -- (an);
            
            \draw[->] (c0) -- (c1);
            \draw[->] (c1) -- (dots3) -- (ck);
            \draw[->] (ck) -- (dots4) -- (cn1);
            \draw[->] (cn1) -- (cn);
            
            \draw[->] (b0) -- (b1);
            \draw[->] (b1) -- (dots5) -- (bk);
            \draw[->] (bk) -- (dots6) -- (bn1);
            \draw[->] (bn1) -- (bn);
            
            \draw[->] (a0) -- (c0);
            \draw[->] (c0) -- (b0);
            \draw[->] (a1) -- (c1);
            \draw[->] (c1) -- (b1);
            \draw[->] (ak) -- (ck);
            \draw[->] (ck) -- (bk);
            \draw[->] (an1) -- (cn1);
            \draw[->] (cn1) -- (bn1);
            \draw[->] (an) --(cn);
            \draw[->] (cn) --(bn);
            
            \draw[->,smooth, purple] plot coordinates {
              (an.south west)
              ($(an1)+(0.2,-0.6)$)
              ($(a1) + (0.1,-0.6)$)
              (b0.north east)
            };
            
        \end{tikzpicture}
    \end{center}
\end{theorem}

\begin{remark}[Generalization of \Cref{thm: interlacing diag}]
    \Cref{thm: interlacing diag} indeed holds for the following general case. 
    However, we omit the proof as it is not needed for the purpose of this paper.
    Let $T\subseteq \{1,\dots , n\}$ be non-empty with minimal element $\min(T) = m$, and let $m \in S\subseteq T$ be isolated.
    Then, the following is an interlacing diagram:
    
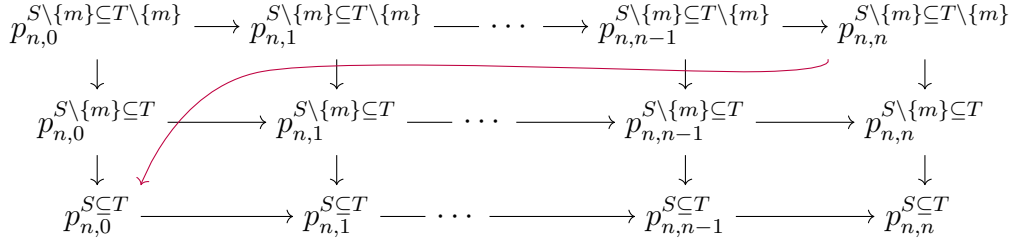
\begin{center}
        \begin{tikzpicture}[node distance=0.4cm and 0.6cm]
            \node (a0) at (0,0) {$\p{n}{0}{S\setminus\{m\}}{T\setminus \{m\}}$};
            \node (a1)  [right=of a0] {$\p{n}{1}{S\setminus\{m\}}{T\setminus \{m\}}$};
            \node (dots1)   [right=of a1]   {$\cdots$};
            \node (an1)  [right=of dots1] {$\p{n}{n-1}{S\setminus\{m\}}{T\setminus \{m\}}$};
            \node (an)  [right=of an1] {$\p{n}{n\phantom{-1}}{S\setminus\{m\}}{T\setminus \{m\}}$};
            
            \node (c0) [below=of a0] {$\p{n}{0}{S\setminus\{m\}}{T}$};
            \node (c1) [below=of a1] {$\p{n}{1}{S\setminus\{m\}}{T}$};
            \node (dots3)   [right=of c1]   {$\cdots$};
            \node (cn1)  [below=of an1] {$\p{n}{n-1}{S\setminus\{m\}}{T}$};
            \node (cn)  [below=of an] {$\p{n}{n\phantom{-1}}{S\setminus\{m\}}{T}$};
            
            \node (b0) [below=of c0] {$\p{n}{0}{S}{T}$};
            \node (b1) [below=of c1] {$\p{n}{1}{S}{T}$};
            \node (dots5)   [right=of b1]   {$\cdots$};
            \node (bn1)  [below=of cn1] {$\p{n}{n-1}{S}{T}$};
            \node (bn)  [below=of cn] {$\p{n}{n\phantom{-1}}{S}{T}$};
            
            \draw[->] (a0) -- (a1);
            \draw[->] (a1) -- (dots1) -- (an1);
            \draw[->] (an1) -- (an);
            
            \draw[->] (c0) -- (c1);
            \draw[->] (c1) -- (dots3) -- (cn1);
            \draw[->] (cn1) -- (cn);
            
            \draw[->] (b0) -- (b1);
            \draw[->] (b1) -- (dots5) -- (bn1);
            \draw[->] (bn1) -- (bn);
            
            \draw[->] (a0) -- (c0);
            \draw[->] (c0) -- (b0);
            \draw[->] (a1) -- (c1);
            \draw[->] (c1) -- (b1);
            \draw[->] (an1) -- (cn1);
            \draw[->] (cn1) -- (bn1);
            \draw[->] (an) --(cn);
            \draw[->] (cn) --(bn);
            
            \draw[->,smooth, purple] plot coordinates {
              (an.south west)
              ($(an1)+(1,-0.6)$)
              ($(a1) + (-1,-0.6)$)
              (b0.north east)
            };
            
        \end{tikzpicture}
    \end{center}
    This implies that
    $( \p{n}{0}{S}{T} \,, \p{n}{0}{S}{T} \,, \dots \,, \p{n}{n}{S}{T} )$ is an interlacing sequence for all $n\geq 1$ and all $S\subseteq T \subseteq\{1,\dots , n\}$.
    In particular, all polynomials $\p{n}{k}{S}{T}(x)$ are real-rooted.
\end{remark}

The proof of \Cref{thm: interlacing diag} is based on a successive application of properties of interlacing polynomials.
Before we dive into the details in \Cref{subsec: pf interlacing diag}, we deduce \Cref{thm: main thm chows are rr,thm: dual chow rr}, the main theorems of this paper.

\begin{proof}[Proof of \Cref{thm: main thm chows are rr,,thm: dual chow rr}]
    Let $P$ be a graded, simplicial poset of rank $n$ with positive $h$-vector $(h_0,\dots, h_n)$.
    For $n=1$, the Chow polynomials are given by
    \[
        \Chow_\Ptop(x) = \Chow_{\Ptop^\ast}(x) = 1+x \quad \text{and} \quad \Chowaug_\Ptop(x) = 1+(m+2)x +x^2 \,,
    \]
    where $m = \# \set{p\in P}{\rk(p) = 1}$. 
    Both polynomials are real-rooted.
    For $n\geq 2$ and $T\in \{[1,n-1], [2,n],[1,n]\}$, the sequence 
    \[
        \left( \pSempty{n}{0}{T}\,, \pSempty{n}{1}{T}\,, \ \dots \,, \pSempty{n}{n}{T} \right)
    \]
    is interlacing by \Cref{thm: interlacing diag}.
    By \Cref{lma: interlacing sums}, any nonnegative linear combination of these polynomials, 
    in particular
    \[
        \sum_{k=1}^{n} h_k \cdot \pSempty{n}{k}{T} (x) \quad \text{ is real-rooted.}
    \]
    As shown in \Cref{lma: gamma Chow via abc}, 
    this polynomial equals
    \begin{equation*}
        \sum_{k=1}^{n} h_k \cdot \pSempty{n}{k}{T} (x)
        = \begin{cases}
            \gamma(\Chow_{\Ptop} ; x ) & \text{if } T=[1,n-1]\,, \\
            \gamma( \Chow_{\Ptop^\ast} ; x ) & \text{if } T=[2,n]\,, \\
            \gamma(\Chowaug_{\Ptop} ; x ) & \text{if } T=[1,n]\,.
        \end{cases}
    \end{equation*}
    \Cref{lma: gamma exp rr} completes the proof for the real-rootedness.

    \smallskip
    Using the pink diagonal arrow in \Cref{thm: interlacing diag}, we also get the interlacing sequence 
    \[
        \left( h_0 \cdot \pSempty{n}{0}{[2,n]}\,, h_1 \cdot \pSempty{n}{1}{[2,n]}\,, \ \dots \,, h_n\cdot \pSempty{n}{n}{[2,n]} \,, 
        h_0 \cdot \p{n}{0}{\{1\}}{[1,n]}\,, h_1 \cdot \p{n}{1}{\{1\}}{[1,n]}\,, \ \dots \,, h_n \cdot \p{n}{n}{\{1\}}{[1,n]} \right) \,.
    \]
     The sum of the first half of this sequence interlaces the sum of the second half. 
     Moreover,
     \[
        \sum_{k=0}^n h_k \cdot \pSempty{n}{k}{[2,n]} \quadinterlace
        \sum_{k=0}^n h_k \cdot \bigl( \pSempty{n}{k}{[2,n]} + \p{n}{k}{\{1\}}{[2,n]} \bigr) 
        = \sum_{k=0}^n h_k \cdot \pSempty{n}{k}{[1,n]} \quadinterlace
        \sum_{k=0}^n h_k \cdot \p{n}{k}{\{1\}}{[2,n]} \,,
     \]
     that is, $\gamma( \Chow_{\Ptop^\ast} ; x )$ interlaces $\gamma(\Chowaug_{\Ptop} ; x )$.
    Using \Cref{prop: f interlaces iff gamma}, this concludes the proof.
\end{proof}

\begin{remark}[Eulerian polynomials]
    The Eulerian polynomial $A_n(x)$ and the binomial Eulerian polynomial $\widetilde{A}_n(x)$
    are defined by
    \[
        A_n(x) = \sum_{w\in \Sym_n} x^{\des(w)}
        \quad \text{and} \quad
        \widetilde{A}_n(x) = 1 + \sum_{i=0}^n \binom{n}{i} A_n(x) \,.
    \]
    The $\gamma$-expansions of $A_n(x)$ was first described in \cite{FoataSchutzenberger} as
    \[
        A_n(x) = \sum_{i \geq 0} \gamma_{n,i} \ x^i (1+x)^{n-1-2i} \,,
    \]
    where $\gamma_{n,i}$ is the number of permutations $w\in \Sym_{n}$, such that $\Des(w)\subset \{1,\dots , n-2\}$ is an isolated set of size $i$.
    A similar $\gamma$-expansions for $\widetilde{A}_n(x)$ was given in \cite{ShareshianWachs}:
    \[
        \widetilde{A}_n(x) = \sum_{k \geq 0} \tilde{\gamma}_{n,i} \ x^i (1+x)^{n-2i} \,,
    \]
    where $\tilde{\gamma}_{n,i}$ is the number of permutations $w\in \Sym_{n}$, such that $\Des(w)\subset \{1,\dots , n-1\}$ is isolated and has size $i$.
    
    \smallskip
    The polynomials $\pSempty{n}{k}{T}$ thus provide a natural decomposition of $A_{n+1}$ and $\widetilde{A}_{n+1}$ into palindromic polynomials, that is
    \begin{align*}
        A_{n+1}(x) =& \sum_{k=0}^{n} (1+x)^{n} \cdot \pSempty{n}{k}{[1,n-1]}\biggl( \frac{x}{(1+x)^2} \biggr)
        \,, \text{ and} \\
        \widetilde{A}_{n+1}(x) =& \sum_{k=0}^n (1+x)^{n+1} \cdot \pSempty{n}{k}{[1,n]}\biggl( \frac{x}{(1+x)^2} \biggr) \,.
    \end{align*}
    The real-rootedness of $A_{n+1}$ and $\widetilde{A}_{n+1}$ has been proven in various ways, see, for example, \cite{HAGLUND201938} for a uniform approach.
    \Cref{thm: interlacing diag} provides yet another proof for their real-rootedness.
    For the Boolean lattice $B_{n+1}$ of rank $n+1$m, we moreover recover
    \[
        \Chow_{B_{n+1}} (x) = A_{n+1}(x)
        \quad \text{and} \quad
        \Chowaug_{B_{n+1}}(x) = \widetilde{A}_{n+1}(x)
    \]
    since the $h$-vector of $B_{n+1} - \{\hat{1}\}$ is $(1,\dots , 1)$.
    The first identity was proven in \cite{HAMPE2017578}, the second in \cite{Eur_Huh_Larson_2023}.
\end{remark}

\subsection{Proof of \Cref{thm: interlacing diag}}
\label{subsec: pf interlacing diag}

For $n\geq 2$ and $1\in T\subseteq \{1,\dots ,n \}$,
let $D_n(T)$ denote the diagram given in \Cref{thm: interlacing diag},

\begin{equation}
    \begin{tikzpicture}[node distance=0.4cm and 0.6cm, baseline=(current bounding box.center)]
        \node (a0) at (0,0) {$\pSempty{n}{0}{T\setminus \{1\}}$};
        \node (a1)  [right=of a0] {$\pSempty{n}{1}{T\setminus \{1\}}$};
        \node (dots1)   [right=of a1]   {$\cdots$};
        \node (ak)   [right=of dots1]   {$\pSempty{n}{k}{T\setminus \{1\}}$};
        \node (dots2)  [right=of ak]   {$\cdots$};
        \node (an1)  [right=of dots2] {$\pSempty{n}{n-1}{T\setminus \{1\}}$};
        \node (an)  [right=of an1] {$\pSempty{n}{n}{T\setminus \{1\}}$};
        
        \node (c0) [below=of a0] {$\pSempty{n}{0}{T}$};
        \node (c1) [below=of a1] {$\pSempty{n}{1}{T}$};
        \node (dots3)   [right=of c1]   {$\cdots$};
        \node (ck)   [below=of ak]   {$\pSempty{n}{k}{T}$};
        \node (dots4)  [right=of ck]   {$\cdots$};
        \node (cn1)  [below=of an1] {$\pSempty{n}{n-1}{T}$};
        \node (cn)  [below=of an] {$\pSempty{n}{n}{T}$};
        
        \node (b0) [below=of c0] {$\p{n}{0}{\{1\}}{T}$};
        \node (b1) [below=of c1] {$\p{n}{1}{\{1\}}{T}$};
        \node (dots5)   [right=of b1]   {$\cdots$};
        \node (bk)   [below=of ck]   {$\p{n}{k}{\{1\}}{T}$};
        \node (dots6)  [right=of bk]   {$\cdots$};
        \node (bn1)  [below=of cn1] {$\p{n}{n-1}{\{1\}}{T}$};
        \node (bn)  [below=of cn] {$\p{n}{n}{\{1\}}{T}$};
        
        \draw[->] (a0) -- (a1);
        \draw[->] (a1) -- (dots1) -- (ak);
        \draw[->] (ak) -- (dots2) -- (an1);
        \draw[->] (an1) -- (an);
        
        \draw[->] (c0) -- (c1);
        \draw[->] (c1) -- (dots3) -- (ck);
        \draw[->] (ck) -- (dots4) -- (cn1);
        \draw[->] (cn1) -- (cn);
        
        \draw[->] (b0) -- (b1);
        \draw[->] (b1) -- (dots5) -- (bk);
        \draw[->] (bk) -- (dots6) -- (bn1);
        \draw[->] (bn1) -- (bn);
        
        \draw[->] (a0) -- (c0);
        \draw[->] (c0) -- (b0);
        \draw[->] (a1) -- (c1);
        \draw[->] (c1) -- (b1);
        \draw[->] (ak) -- (ck);
        \draw[->] (ck) -- (bk);
        \draw[->] (an1) -- (cn1);
        \draw[->] (cn1) -- (bn1);
        \draw[->] (an) --(cn);
        \draw[->] (cn) --(bn);
        
        \draw[->,smooth, purple] plot coordinates {
          (an.south west)
          ($(an1)+(0.2,-0.6)$)
          ($(a1) + (0.1,-0.6)$)
          (b0.north east)
        };
        
    \end{tikzpicture} \label{eq: interlacing diag}
\end{equation}

\Cref{thm: interlacing diag} claims that $D_n([1,n])$ as well as $D_n([1,n-1])$ are interlacing diagrams.
We proceed by induction on $n$. 
To clearly present the arguments, we split them into its two base cases (\Cref{lma: base case 1,lma: base case 2}), and five statements for the induction step (see \Cref{lma: induction step}).

\begin{lemma}[Base case for $T=\{1\}$]\label{lma: base case 1}
    The diagram $D_2(\{1\})$ is an interlacing diagram.
\end{lemma}
\begin{proof}
    The diagram $D_2(\{1\})$ is
    \begin{center}
    \begin{tikzpicture}[node distance=1.2cm and 2cm] 

      \begin{scope}[local bounding box=lhs]
        \node (a0) at (0,2.4) {$\pSempty{2}{0}{\emptyset}$};
        \node (a1) at (2.2,2.4) {$\pSempty{2}{1}{\emptyset}$};
        \node (a2) at (4.4,2.4) {$\pSempty{2}{2}{\emptyset}$};
        
        \node (c0) at (0,1.2) {$\pSempty{2}{0}{\{1\}}$};
        \node (c1) at (2.2,1.2) {$\pSempty{2}{1}{\{1\}}$};
        \node (c2) at (4.4,1.2) {$\pSempty{2}{2}{\{1\}}$};
        
        \node (b0) at (0,0) {$\p{2}{0}{\{1\}}{\{1\}}$};
        \node (b1) at (2.2,0) {$\p{2}{1}{\{1\}}{\{1\}}$};
        \node (b2) at (4.4,0) {$\p{2}{2}{\{1\}}{\{1\}}$};
        
        \draw[->] (a0) -- (a1);
        \draw[->] (a1) -- (a2);
        \draw[->] (c0) -- (c1);
        \draw[->] (c1) -- (c2);
        \draw[->] (b0) -- (b1);
        \draw[->] (b1) -- (b2);
    
        \draw[->] (a0) -- (c0);
        \draw[->] (a1) -- (c1);
        \draw[->] (a2) -- (c2);
        \draw[->] (c0) -- (b0);
        \draw[->] (c1) -- (b1);
        \draw[->] (c2) -- (b2);
    
        \draw[->,smooth, purple] plot coordinates {
          (a2.south west)
          ($(a2.south west)+(-0.4,-0.2)$)
          ($(a1)+(-0.5,-0.7)$)
          (b0.north east)
        };
      \end{scope}
    
      \node at ($(lhs.east)+(1.0,0)$) {$=$};

      \begin{scope}[xshift=8cm] 
        \node (A0) at (0,2.4) {$1$};
        \node (A1) at (2,2.4) {$0$};
        \node (A2) at (4,2.4) {$0$};
        
        \node (C0) at (0,1.2) {$1$};
        \node (C1) at (2,1.2) {$x$};
        \node (C2) at (4,1.2) {$x$};
        
        \node (B0) at (0,0) {$0$};
        \node (B1) at (2,0) {$x$};
        \node (B2) at (4,0) {$x$};
        
        \draw[->] (A0) -- (A1);
        \draw[->] (A1) -- (A2);
        \draw[->] (C0) -- (C1);
        \draw[->] (C1) -- (C2);
        \draw[->] (B0) -- (B1);
        \draw[->] (B1) -- (B2);
    
        \draw[->] (A0) -- (C0);
        \draw[->] (A1) -- (C1);
        \draw[->] (A2) -- (C2);
        \draw[->] (C0) -- (B0);
        \draw[->] (C1) -- (B1);
        \draw[->] (C2) -- (B2);
    
        \draw[->,smooth, purple] plot coordinates {
          (A2.south west)
          ($(A2.south west)+(-0.4,-0.2)$)
          ($(A1)+(-0.5,-0.7)$)
          (B0.north east)
        };
      \end{scope}
    \end{tikzpicture}
    \end{center}
    and the roots of the polynomials interlace along every path.
\end{proof}

\begin{lemma}[Base case for $T=\{1,2\}$]\label{lma: base case 2}
    The diagram $D_2(\{1,2\})$ is an interlacing diagram.
\end{lemma}
\begin{proof}
    The diagram $D_2(\{1,2\})$ is
    \begin{center}
    \begin{tikzpicture}[node distance=1.2cm and 2cm] 

      \begin{scope}[local bounding box=lhs]
        \node (a0) at (0,2.4) {$\pSempty{2}{0}{\{2\}}$};
        \node (a1) at (2.2,2.4) {$\pSempty{2}{1}{\{2\}}$};
        \node (a2) at (4.4,2.4) {$\pSempty{2}{2}{\{2\}}$};
        
        \node (c0) at (0,1.2) {$\pSempty{2}{0}{\{1,2\}}$};
        \node (c1) at (2.2,1.2) {$\pSempty{2}{1}{\{1,2\}}$};
        \node (c2) at (4.4,1.2) {$\pSempty{2}{2}{\{1,2\}}$};
        
        \node (b0) at (0,0) {$\p{2}{0}{\{1\}}{\{1,2\}}$};
        \node (b1) at (2.2,0) {$\p{2}{1}{\{1\}}{\{1,2\}}$};
        \node (b2) at (4.4,0) {$\p{2}{2}{\{1\}}{\{1,2\}}$};
        
        \draw[->] (a0) -- (a1);
        \draw[->] (a1) -- (a2);
        \draw[->] (c0) -- (c1);
        \draw[->] (c1) -- (c2);
        \draw[->] (b0) -- (b1);
        \draw[->] (b1) -- (b2);
    
        \draw[->] (a0) -- (c0);
        \draw[->] (a1) -- (c1);
        \draw[->] (a2) -- (c2);
        \draw[->] (c0) -- (b0);
        \draw[->] (c1) -- (b1);
        \draw[->] (c2) -- (b2);
    
        \draw[->,smooth, purple] plot coordinates {
          (a2.south west)
          ($(a2.south west)+(-0.4,-0.2)$)
          ($(a1)+(-0.5,-0.7)$)
          (b0.north east)
        };
      \end{scope}
    
      \node at ($(lhs.east)+(1.0,0)$) {$=$};

      \begin{scope}[xshift=8cm] 
        \node (A0) at (0,2.4) {$x+1$};
        \node (A1) at (2,2.4) {$x$};
        \node (A2) at (4,2.4) {$0$};
        
        \node (C0) at (0,1.2) {$x+1$};
        \node (C1) at (2,1.2) {$2x$};
        \node (C2) at (4,1.2) {$x$};
        
        \node (B0) at (0,0) {$0$};
        \node (B1) at (2,0) {$x$};
        \node (B2) at (4,0) {$x$};
        
        \draw[->] (A0) -- (A1);
        \draw[->] (A1) -- (A2);
        \draw[->] (C0) -- (C1);
        \draw[->] (C1) -- (C2);
        \draw[->] (B0) -- (B1);
        \draw[->] (B1) -- (B2);
    
        \draw[->] (A0) -- (C0);
        \draw[->] (A1) -- (C1);
        \draw[->] (A2) -- (C2);
        \draw[->] (C0) -- (B0);
        \draw[->] (C1) -- (B1);
        \draw[->] (C2) -- (B2);
    
        \draw[->,smooth, purple] plot coordinates {
          (A2.south west)
          ($(A2.south west)+(-0.4,-0.2)$)
          ($(A1)+(-0.5,-0.7)$)
          (B0.north east)
        };
      \end{scope}
    \end{tikzpicture}
    \end{center}
    and the roots of the polynomials interlace along every path.
\end{proof}

Now we have proven the base cases, we move to the induction step.
The following lemma shows us what we need to verify in the induction step.

\begin{lemma}\label{lma: induction step}
    Let $T=[1,n]$ or $T=[1,n-1]$ for $n\geq 3$. 
    The diagram~$D_{n}(T)$, as given as in \eqref{eq: interlacing diag}, is an interlacing diagram if and only if the following five statements hold:
    \begin{align}
        &\biggl( \pSempty{n}{0}{T\setminus \{1\}} \,, \pSempty{n}{1}{T\setminus \{1\}}\,, \dots \,, \pSempty{n}{n}{T\setminus \{1\}} \biggr)
            \text{ is an interlacing sequence,}\label{eq: ainter} \\[5pt]
        &\biggl(\pSempty{n}{0}{T} \,, \pSempty{n}{1}{T}\,, \dots \,, \pSempty{n+1}{n+1}{T} \biggr)
            \text{ is an interlacing sequence,} \label{eq: cinter} \\[5pt]
        &\biggl(\p{n+1}{0}{\{1\}}{T} \,, \p{n+1}{1}{\{1\}}{T}\,, \dots \,, \p{n+1}{n+1}{\{1\}}{T} \biggr)
            \text{ is an interlacing sequence,} \label{eq: binter} \\[5pt]
        &\biggl(\pSempty{n}{k}{T\setminus \{1\}} \,, \pSempty{n}{k}{T} \,, \p{n}{k}{{\{1\}}}{T} \biggr) 
            \text{ is an interlacing sequence for all } 0\leq k \leq n \,, 
            \text{ and} \label{eq: acbinter} \\[5pt]
        &\pSempty{n}{n-1}{T\setminus \{1\}} \spaceinterlace \p{n}{0}{\{1\}}{T} \,. \label{eq: abinter}
    \end{align}
\end{lemma}
\begin{proof}
    All five statements follow immediately when $D_n(T)$ is an interlacing diagram. 
    Thus, suppose that \eqref{eq: ainter}--\eqref{eq: abinter} hold.
    The only polynomials that are zero in the diagram $D_{n}(T)$ are  
    \[
        \pSempty{n}{n}{T\setminus \{1\} }(x) \equiv 0
        \quad \text{and} \quad
        \p{n}{0}{\{1\} }{T} (x) \equiv 0 \,.
    \]
    Any other polynomial in $D_{n}(T)$ is a non-empty sum of polynomials with nonnegative coefficients of which at least one is not zero.

    \smallskip
    Consider any path of polynomials $(f_1 \rightarrow f_2 \rightarrow \cdots \rightarrow f_m)$ in $D_n(T)$ that contains neither~$\pSempty{n}{n}{T\setminus \{1\}}$ nor~$\p{n}{0}{\{1\} }{T}$.
    Without loss of generality, the first polynomial in this path is the top left one,~$f_1=\pSempty{n}{0}{T\setminus \{1\} }$ and the last polynomial is the bottom right one,~$f_m = \p{n}{0}{\{1\} }{T}$.
    Then, any two adjacent polynomials $f_i \rightarrow f_{i+1}$ interlace as $f_i \interlace f_{i+1}$, since they are either 
    \begin{itemize}
        \item in the same row of the diagram $D_{n}(T)$ (interlace by \eqref{eq: ainter}, \eqref{eq: cinter} or \eqref{eq: binter}), or
        \item in the same column of the diagram $D_n(T)$ (interlace by \eqref{eq: acbinter}), or
        \item $f_i = \pSempty{n}{n-1}{T\setminus \{1\}}$ and $f_{i+1} = \p{n}{0}{\{1\}}{T}$ (interlace by \eqref{eq: abinter}).
    \end{itemize}
    Moreover, the first and last polynomial in the path satisfy the identities
    \begin{align*}
        f_1 &= \pSempty{n}{0}{T\setminus \{1\} } 
            = \pSempty{n}{0}{T} - \p{n}{0}{\{1\}}{T\setminus \{1\} } 
            = \pSempty{n}{0}{T} \quad \text{ and}\\
        f_m &= \p{n}{n}{\{1\}}{T} 
            = \pSempty{n}{n}{T} - \pSempty{n}{n}{T\setminus \{1\} } 
            = \pSempty{n}{n}{T} \,.
    \end{align*}
    Thus, $f_1 \interlace f_m$ since the middle row of $D_n(T)$ is interlacing by \eqref{eq: cinter}. Since all polynomials in the path are nonzero, they form an interlacing sequence~$(f_1 \interlace f_2 \interlace \cdots \interlace f_m)$ by \Cref{lma: interlacing seq}.
    This proves that the diagram $D_n(T)$ with zeros omitted,
    \begin{center}
    \begin{tikzpicture}[node distance=0.4cm and 0.6cm]
        \node (a0) at (0,0) {$\pSempty{n}{0}{T\setminus \{1\}}$};
        \node (a1)  [right=of a0] {$\pSempty{n}{1}{T\setminus \{1\}}$};
        \node (dots1)   [right=of a1]   {$\cdots$};
        \node (an1)  [right=of dots1] {$\pSempty{n}{n-1}{T\setminus \{1\}}$};
        \node (an)  [right=of an1] {$\phantom{\pSempty{n}{n}{T\setminus \{1\}}}$};
        
        \node (c0) [below=of a0] {$\pSempty{n}{0}{T}$};
        \node (c1) [below=of a1] {$\pSempty{n}{1}{T}$};
        \node (dots3)   [right=of c1]   {$\cdots$};
        \node (cn1)  [below=of an1] {$\pSempty{n}{n-1}{T}$};
        \node (cn)  [below=of an] {$\pSempty{n}{n}{T}$};
        
        \node (b0) [below=of c0] {$\phantom{\p{n}{0}{\{1\}}{T}}$};
        \node (b1) [below=of c1] {$\p{n}{1}{\{1\}}{T}$};
        \node (dots5)   [right=of b1]   {$\cdots$};
        \node (bn1)  [below=of cn1] {$\p{n}{n-1}{\{1\}}{T}$};
        \node (bn)  [below=of cn] {$\p{n}{n}{\{1\}}{T}$};
        
        \draw[->] (a0) -- (a1);
        \draw[->] (a1) -- (dots1) -- (an1);
        
        \draw[->] (c0) -- (c1);
        \draw[->] (c1) -- (dots3) -- (cn1);
        \draw[->] (cn1) -- (cn);
        
        \draw[->] (b1) -- (dots5) -- (bn1);
        \draw[->] (bn1) -- (bn);
        
        \draw[->] (a0) -- (c0);
        \draw[->] (a1) -- (c1);
        \draw[->] (c1) -- (b1);
        \draw[->] (an1) -- (cn1);
        \draw[->] (cn1) -- (bn1);
        \draw[->] (cn) --(bn);
        
        \draw[->,smooth, purple] plot coordinates {
          (an1.south west)
          (b1.north east)
        };
    \end{tikzpicture}
\end{center}
    is an interlacing diagram. 
    Adding "0" in the bottom left corner and the top right corner to this diagram does not affect the interlacing properties.
\end{proof}

Assume from now that $D_{n-1}([1,n-2])$ and $D_{n-1}([1,n-1])$ are both interlacing diagrams for fixed $n-1\geq 3$.
The induction step of our induction within the next lemmas shows, that both $D_n([1,n-1])$ and $D_n([1,n])$ are interlacing diagrams.
Recall the recursive formulas for the polynomials in the diagram~$D_n(T)$:
\begin{align*}
        \pSempty{n}{k}{T\setminus\{1\}}(x)
            &= \phantom{ x \cdot \sum_{j = 0}^{k-1} \pSempty{n-1}{j}{(T-1)\setminus\{1\}} (x)
                + } \sum_{j=k}^{n-1} \pSempty{n-1}{j}{T -1} (x) \,, \\
        \pSempty{n}{k}{T } (x)
            &= x \cdot \sum_{j = 0}^{k-1} \pSempty{n-1}{j}{(T-1)\setminus\{1\}} (x) 
                + \sum_{j=k}^{n-1} \pSempty{n-1}{j}{T-1} (x) 
                \quad = \ \pSempty{n}{k}{T\setminus\{1\}}(x) + \p{n}{k}{\{1\}}{T} (x) \,, \\
        \p{n}{k}{\{1\}}{T} (x) 
            &= x \cdot \sum_{j = 0}^{k-1} \pSempty{n-1}{j}{(T-1)\setminus\{1\}} (x) \,.
\end{align*}

The following proofs will therefore work with $D_{n-1}(T-1)$ which is an interlacing diagram by induction hypothesis and contains all polynomials showing up in the recursive formulas.

\begin{lemma}[Top row, \eqref{eq: ainter}]
    Let either $T = [2,n-1]$ or $T=[2,n-2]$.
    Then
    \[
      (\pSempty{n}{0}{T} \,, \pSempty{n}{1}{T}\,, \dots \,, \pSempty{n}{n}{T}) \text{ is interlacing}\,.
    \]
\end{lemma}
\begin{proof}
    Consider the interlacing sequence
    \(
    \left( \pSempty{n-1}{0}{T-1}\interlace \cdots \interlace \pSempty{n-1}{n-1}{T-1} \right) \,,
    \)
    the middle row of the diagram~$D_{n-1}(T-1)$.
    By \Cref{lma: recycle interlacing seq}, the upper partial sums 
    \[
      \pSempty{n}{k}{T} = \sum_{j=k}^{n-1} \pSempty{n-1}{k}{T-1}
    \]
    form an interlacing sequence.
\end{proof}

\begin{lemma}[Middle row, \eqref{eq: cinter}]
    Let either $T = [1,n-2]$ or $T=[1,n-1]$. 
    Then
    \[
      (\pSempty{n}{0}{T} \,, \pSempty{n}{1}{T}\,, \dots \,, \pSempty{n}{n}{T}) \text{ is interlacing}\,.
    \]
\end{lemma}
\begin{proof}
    Consider the interlacing sequences consisting of the top row of $D_{n-1}(T-1)$ and the bottom row.
    By \Cref{lma: recycle interlacing seq,lma: interlacing seq}, the diagram
    \begin{center}
    \resizebox{\linewidth}{!}{%
        \begin{tikzpicture}
            [node distance=0.4cm and 0.6cm]
            
            \node (a0) at (0,0) {$\pSempty{n-1}{0}{(T-1)\setminus \{1\}}$};
            \node (a1)  [right=of a0] {$\pSempty{n-1}{1}{(T-1)\setminus \{1\}}$};
            \node (dots1)  [right=of a1] {$\cdots$};
            \node (an)  [right=of dots1]  {$\pSempty{n-1}{n-1}{(T-1)\setminus \{1\}}$};
            \node (b0)   [right=of an]   {$\p{n-1}{0}{\{1\}}{T-1}$};
            \node (b1)   [right=of b0]   {$\p{n-1}{1}{\{1\}}{T-1}$};
            \node (dots2)   [right=of b1]   {$\cdots$};
            \node (bn)   [right=of dots2]   {$\p{n-1}{n-1}{\{1\}}{T-1}$};
            \node (a) [below=of a1] {$\sum_{k= 0}^{n-1} \pSempty{n-1}{k}{(T-1)\setminus \{1\}}$};
            \node (c) [right=of a] {$\sum_{k= 0}^{n-1} (\pSempty{n-1}{k}{(T-1)\setminus \{1\}} + \p{n-1}{k}{\{1\}}{T-1})$};
            \node (b) [right=of c] {$\sum_{k= 1}^{n-1} \p{n-1}{k}{\{1\}}{T-1}$};
            
            \draw[->] (a0) -- (a1);
            \draw[->] (a1) -- (dots1) -- (an);
            \draw[->] (an) -- (b0);
            \draw[->] (b0) -- (b1);
            \draw[->] (b1) -- (dots2) -- (bn);
            \draw[->] (a0) -- (a);
            \draw[->] (a) -- (an);
            \draw[->] (b0) -- (b);
            \draw[->] (b) -- (bn);
            \draw[->] (a) -- (c);
            \draw[->] (c) -- (b);
            
        \end{tikzpicture}
    }
    \end{center}
    is an interlacing diagram.
    In particular, we obtain
    \[
    \sum_{k= 0}^{n-1} \pSempty{n-1}{k}{(T-1)\setminus \{1\}}
    \quadinterlace \sum_{k= 0}^{n-1} \pSempty{n-1}{k}{T -1}
    \quadinterlace \sum_{k= 1}^{n-1} \p{n-1}{k}{\{1\}}{T-1} \,.
    \]
    We apply \Cref{lma: interlacing sums} by multiplying the leftmost polynomial by $x$.
    This yields
    \[
    \pSempty{n}{0}{T\setminus \{1\}}
    = \sum_{k= 0}^{n-1} \pSempty{n-1}{k}{T -1}
    \quadinterlace x \cdot \sum_{k= 0}^{n-1} \pSempty{n-1}{k}{(T-1)\setminus \{1\}}
    = \p{n}{n}{\{1\}}{T} \,,
    \]
    which recovers $\pSempty{n}{0}{T} \interlace \pSempty{n}{k}{T}$,
    the interlacing of the first entry and last entry in the middle row of $D_n(T)$, 
    since $\pSempty{n}{0}{T} = \pSempty{n}{0}{T\setminus\{1\}}$ 
    and $\pSempty{n}{n}{T} = \p{n}{n}{\{1\}}{T}$.
    
    \smallskip
    It remains to verify the interlacing of adjacent polynomials in the middle row of $D_n(T)$.
    For this, we turn to the interlacing diagram $D_{n-1}(T-1)$ without the bottom row.
    In words, for every $k\leq n$, we have the interlacing relations~$\pSempty{n-1}{i}{(T-1)\setminus \{1\}} \interlace \pSempty{n-1}{\ell}{T -1}$ for $i\leq k \leq \ell$.
    We apply \Cref{lma: interlacing sums} successively to derive several interlacing relations,
    \begin{align*}
    \pSempty{n-1}{i}{(T-1)\setminus \{1\}} 
    &\spaceinterlace \sum_{j = k+1}^{n-1} \pSempty{n-1}{j}{T -1}\,, \qquad 
    \pSempty{n-1}{k}{T -1} \spaceinterlace \sum_{j = k+1}^{n-1} \pSempty{n-1}{j}{T -1} \,, \\
    \sum_{j = 0}^{k-1} \pSempty{n-1}{j}{(T-1)\setminus \{1\}} 
    &\spaceinterlace \pSempty{n-1}{k}{(T-1)\setminus \{1\}} \,, \qquad 
    \sum_{j = 0}^{k-1} \pSempty{n-1}{j}{(T-1)\setminus \{1\}} \spaceinterlace \pSempty{n-1}{k}{T -1} \,.
    \end{align*}
    Using \Cref{lma: interlacing sums}, we obtain the interlacing sequence
    \[
    \biggl( \pSempty{n-1}{k}{T -1} \spaceinterlace \sum_{j = k+1}^{n-1} \pSempty{n-1}{j}{T -1} \spaceinterlace x \cdot \sum_{j = 0}^{k-1} \pSempty{n-1}{j}{(T-1)\setminus \{1\}} \spaceinterlace x \cdot \pSempty{n-1}{k}{(T-1)\setminus \{1\}} \biggr)
    \]
    By \Cref{lma: recycle interlacing seq}, the sum of the first three polynomials interlaces the sum of the last three polynomials.
    This gives $\pSempty{n}{k}{T} \interlace \pSempty{n}{k+1}{T}$ for $0\leq k \leq n$, which completes the proof.
\end{proof}

\begin{lemma}[Bottom row, \eqref{eq: binter}]
    Let either $T = [1,n-1]$ or $T=[1,n]$. 
    Then, 
    \[
    (\pSempty{n}{0}{T} \,, \pSempty{n}{1}{T}\,, \dots \,, \pSempty{n}{n}{T}) \text{ is interlacing}\,.
    \]
\end{lemma}
\begin{proof}
    Consider the lower partial sums of the top row interlacing sequence in the diagram~$D_{n-1}(T-1)$.
    By \Cref{lma: recycle interlacing seq}, these partial sums also form an interlacing sequence,
    \[
    \biggl( \sum_{j< 0} \pSempty{n-1}{j}{(T-1)\setminus \{1\}} \spaceinterlace \sum_{j< 1} \pSempty{n-1}{j}{(T-1)\setminus \{1\}} \spaceinterlace \dots \spaceinterlace \sum_{j< n} \pSempty{n-1}{j}{(T-1)\setminus \{1\}} \biggr)\,.
    \]
    Multiplying each polynomial by $x$ in this sequence will preserve the interlacing relations by \Cref{lma: interlacing sums}. 
    Since~$x\cdot \sum_{j< k} \pSempty{n-1}{j}{(T-1)\setminus \{1\}} = \p{n}{k}{\{1\}}{T}$,
    this sequence equals the bottom row in $D_n(T)$.
\end{proof}

\begin{lemma}[Columns, \eqref{eq: acbinter}]
    Let $T = [1,n-1]$ or $T=[1,n]$.
    Then, 
    \[
    \big(\pSempty{n}{k}{T\setminus \{1\}} \,, \pSempty{n}{k}{T}\,, \p{n}{k}{\{1\}}{T}\big) \text{ is interlacing for all } 0\leq k \leq n \,.
    \]
\end{lemma}
\begin{proof}
    As done before, consider the interlacing diagram $D_{n-1}(T-1)$ without the bottom row.
    Then, for every $k\leq n$, we have the interlacing relations~$\pSempty{n-1}{i}{(T-1)\setminus \{1\}} \interlace \pSempty{n-1}{\ell}{T -1}$ for $i\leq k \leq \ell$.
    Applying \Cref{lma: interlacing sums} successively yields the interlacing relation
    \[
        \sum_{j\geq k} \pSempty{n-1}{j}{T -1} \quadinterlace x\cdot \sum_{j< k} \pSempty{n-1}{0}{(T-1)\setminus \{1\}}\,.
    \] 
    This restates what we want to prove, 
    since~$\pSempty{n}{k}{T -1} = \pSempty{n}{k}{(T-1)\setminus \{1\}} + \p{n}{k}{\{1\}}{T-1}$.
\end{proof}

\begin{lemma}[top right $\interlace$ bottom left, \eqref{eq: abinter} ]
    Let either $T = [1,n-1]$ or $T=[1,n]$.
    Then,
    \[
      \pSempty{n}{n-1}{T\setminus \{1\}} \spaceinterlace \p{n}{1}{\{1\}}{T}\,.
    \]
\end{lemma}
\begin{proof}
    Again, consider the interlacing diagram~$D_{n-1}(T-1)$.
    The polynomial~$\pSempty{n-1}{0}{(T-1)\setminus \{1\}}$ interlaces every polynomial in the diagram~$D_{n-1}(T-1)$, 
    in particular~$\pSempty{n-1}{0}{(T-1)\setminus \{1\}} \interlace \pSempty{n}{n-1}{T\setminus \{1\}}$. 
    Thus,~$\pSempty{n}{n-1}{T\setminus \{1\}} \interlace x \cdot \pSempty{n-1}{0}{T\setminus \{1\}} = \p{n}{1}{\{1\}}{T}$ by \Cref{lma: interlacing sums}.
\end{proof}

\begin{proof}[Proof of \Cref{thm: interlacing diag}]
    We proceed by induction on $n$.
    The two base cases for $n=1$ are given in \Cref{lma: base case 1,lma: base case 2}.
    The five lemmas after \Cref{lma: induction step} prove the induction step by \Cref{lma: induction step}.
\end{proof}

\bibliographystyle{amsalpha}
\bibliography{bibliography}
    
\end{document}